\documentclass[a4paper,11pt]{article}

\usepackage{amsmath}
\usepackage{amsfonts}
\usepackage{amsthm}
\usepackage{amssymb}
\usepackage{thm-restate}
\usepackage{enumerate}
\usepackage{framed}
\usepackage{mathtools}
\usepackage{xifthen}
\usepackage{float}
\usepackage{subcaption}
\usepackage{mathdots}
\usepackage{comment}

\newtheorem{theorem}{Theorem}[section]

\newtheorem{lemma}[theorem]{Lemma}

\newtheorem{lem}[theorem]{Lemma}

\newtheorem{cor}[theorem]{Corollary}

\newcommand{\X}{\bar{X}}
\newcommand{\es}{\emptyset}
\newcommand{\sm}{\setminus}

\theoremstyle{definition}

\title{Catching a mouse on a tree}
\author{Vytautas Gruslys\footnote{Department of Pure Mathematics and Mathematical Statistics, Centre for Mathematical Sciences, Wilberforce Road, Cambridge, CB3 0WB, UK.} \and Ar\`es M\'eroueh\footnotemark[\value{footnote}] }

\newcommand{\N}{\mathbb{N}}
\newcommand{\s}{\mathcal{S}}

\begin{document}
\maketitle

\begin{abstract}
In this paper we consider a pursuit-evasion game on a graph.
A team of cats, which may choose any vertex of the graph at any turn,
tries to catch an invisible mouse, which is constrained to moving along
the vertices of the graph. Our main focus shall be on trees. We prove
that $\lceil (1/2)\log_2(n)\rceil$ cats can always catch a mouse on a
tree of order $n$ and give a collection of trees where the mouse can avoid
being caught by $ (1/4 - o(1))\log_2(n)$ cats.  
\end{abstract}

\section{Introduction}
We consider the following game played on a graph $G$. A mouse and a team of $r$ cats take turns choosing vertices of the graph. On the mouse's turn, it must move to a vertex adjacent to its current position. On the cats' turn, each cat may choose any vertex of the graph as its next position. If the mouse can evade the cats indefinitely on the graph, then $G$ is an $r$-mouse-win. Otherwise, $G$ is an $r$-cat-win. 

This game was introduced independently for one cat playing against one
mouse (i.e. $r=1$) by Haslegrave \cite{haslegrave} and by Britnell and
Wildon \cite{wildon}. They proved the following.

\begin{theorem}[Haslegrave \cite{haslegrave}, Britnell and Wildon\cite{wildon}]
\label{1catwin}
A graph $G$ is a $1$-cat-win if and only if it is a tree not containing
the tree $H$, where $H$ is the tree on ten vertices which is the
3-subdivision of a star on four vertices (in other words, $H$ is obtained
by replacing each edge of a star on four vertices by a path on four vertices).
\end{theorem}

Very recently Abramoskaya, Fomin, Golovach and Pilipczuk \cite{hunterrabbit}
extended this problem to an arbitrary number of cats. They defined $h(G)$ to
be the minimal integer such that $G$ is an $h(G)$-cat-win and proved the following two theorems.  

\begin{theorem}[Abramoskaya, Fomin, Golovach and Philipczuk \cite{hunterrabbit}]
\label{grid}
Let $G$ be an $n\times m$ grid. Then
$h(G) = \left\lfloor\frac{\min(n,m)}{2}\right\rfloor+1$. 
\end{theorem}

\begin{theorem}[Abramoskaya, Fomin, Golovach and Philipczuk \cite{hunterrabbit}]
\label{trees}
There exists a constant $c_1$ such that for any tree $T$ of order $n$,
$h(T)\leq c_1\log(n)$. On the other hand, there is a constant $c_2$ such
that for any $n$, there exists a tree $T$ of order $n$ with
$h(T)\geq c_2\log(n)/\log(\log(n))$.
\end{theorem}

All the logarithms in this paper are in base 2 unless otherwise stated.
Our main results are the following.

\begin{restatable}{theorem}{upperboundrestate} \label{upperbound}
Let $T$ be a tree of order $n$. Then $h(T)\leq \lceil(1/2)\log(n)\rceil$.  
\end{restatable}

\begin{restatable}{theorem}{lowerboundrestate} \label{lowerbound}
For any $\epsilon > 0$ and any sufficiently large $n$ there is a tree
$T$ of order $n$ such that $h(T) \ge (1/4 - \epsilon) \log(n)$.
\end{restatable}

Thus, if we let 
$$g(n) = \max\{h(T): T \text{ is a tree of order } n \} $$
then the growth rate of $g$ is determined up to a constant. Namely,
we prove that $g(n) = \Theta(\log(n))$. It would be interesting to determine
$g(n)$ asymptotically.

The organization of this paper is as follows. In Section 1 we first
prove a basic upper bound of $\lceil \log(n) \rceil$ on $g(n)$ and then
we go on to prove Theorem \ref{upperbound}. In Section 2 we first prove
that $g(n) = \Omega(\log(n))$ thereby substantially improving the lower
bound of \cite{hunterrabbit}, and then we go on to prove
Theorem \ref{lowerbound}.  
 
\section{Upper bound}

\subsection{Basic upper bound}

\begin{lemma}
\label{centre}
Let $T$ be a tree. Then there exists $v\in V(T)$ such that $T-v$ is a forest each component of which has order no more than $\lceil (n-1)/2\rceil $. Such a vertex is called a \emph{centre} of $T$. 
\end{lemma}
\begin{proof}
The proof is by induction on $n = |V(T)|$. For $n=1,2$ the claim is clear. 

Suppose $n\geq 3$ and $n$ is even, i.e. $n=2k$ for some $n\in \N$. $T$ contains a leaf $w$. By induction there exists $v\in V(T-w)$ such that $T-w-v$ is a forest with components $T_1$, $T_2$,\ldots, $T_l$ such that each one of them has order no more than $k-1$. If $w$ is joined to $v$ then $\{w \},T_1,T_2,\ldots, T_l$ are the components of $T-v$ and since $ 1 \leq \lceil (n-1)/2\rceil$ we are done. Otherwise $w$ is joined to a vertex of one of the $T_i$'s, without loss of generality $T_1$. Then $T_1\cup \{w\}, T_2,\ldots, T_l$ are the components of $T-v$ and since $k = \lceil (n-1)/2\rceil$ we are done. 

Now suppose $n\geq 3$ is odd, i.e. $n= 2k+1$ for some $k\in \N$. As before we let $w$ be a leaf of $T$ and let $v$ be a vertex such that $T-w-v$ is a forest with components $T_1$, $T_2$,\ldots, $T_l$ of size no more than $k$ each. Again if $w$ is joined to $v$ we are done. Otherwise we may assume without loss of generality that it is joined to a vertex of $T_1$. If $|T_1|\leq k-1$ then $T_1\cup \{w\}, T_2,\ldots, T_l$ are the components of $T-v$ and since $k = \lceil (n-1)/2\rceil$ we are done. So assume that $|T_1| = k$. Then notice that $\sum_{i\geq 2}|T_i| =k-1$. Now $v$ is joined to a unique element of $T_1$, call it $y$. Then $T-y$ is made of two parts $T_1 - y$ and $\left(\cup_{i\geq 2}T_i\right)\cup\{ v\}$ with no edge between them and having order no more than $k$ each, hence the result.
\end{proof}

\begin{lemma}
\label{basicupperbound}
Let $T$ be a tree of order $n$. Then $\lceil\log_2(n)\rceil$ cats have a winning strategy on $T$. 
\end{lemma} 

\begin{proof}
The proof is by induction on $n$. It is clear that one cat can catch the mouse on trees of order no more than 9 by Theorem \ref{1catwin}. So let $T$ be a tree of order $n\geq 10$. Suppose that the induction hypothesis holds for all trees of order less than $n$. By Lemma \ref{centre} $T$ contains a centre $v$. Let $T_1$, $T_2$,\ldots $T_l$ be the components of $T-v$. By induction hypothesis for each $i$ there is a strategy $\mathcal{S}_i$ for $r=\lceil\log_2(\lceil (n-1)/2\rceil)\rceil$ cats to win on $T_i$. We form a new strategy $\mathcal{S}$ for $1r+1$ cats to win on $T$ as follows. Let $C_1, C_2,\ldots, C_{r+1}$ be the cats at our disposal to catch the mouse. At every step $C_1$ chooses $v$. This ensure that the mouse can never travel from one component of $T-v$ to another. In the meantime, the remaining cats run the strategy $\mathcal{S}_1$ on $T_1$, then $\mathcal{S}_2$ on $T_2$, etc, until $\mathcal{S}_i$ is run on $T_i$ for each $i$. We claim that at the end the mouse is caught; indeed assuming it start in component $T_j$ then it always remains in this component thanks to $C_1$ being always posted on $v$. Then when the remaining cats run $\mathcal{S}_j$ on $T_j$ they must catch the mouse. Now $1+r = 1 + \lceil\log_2(\lceil (n-1)/2\rceil)\rceil \leq 1 + \lceil \log_2(n/2)\rceil = 1 + \lceil\log_2(n) - 1 \rceil = \lceil \log_2(n) \rceil$, hence the lemma.    
\end{proof}

\subsection{Improved upper bound}

We begin by defining a variant of the game for bipartite graphs. Let $G$ be a bipartite graph with vertex classes $V_1$ and $V_2$. In the $V_1$-variant of the game, the mouse if forced to choose a vertex from $V_1$ in its first turn. Otherwise, the game is the same. Thus the essential difference between this new game and the standard one is that  after every turn of the mouse the cats know which vertex class of $G$ the mouse lies in since $G$ is bipartite. 

\begin{lemma}
\label{equivalence}
Let $G$ be a bipartite graph with vertex classes $V_1$ and $V_2$. Then $l$ cats have a winning strategy on $G$ for the $V_1$-game if and only if they have a winning strategy for the standard game.
\end{lemma} 

\begin{proof}
One direction is obvious: if the cats have a winning strategy for the standard game, the same strategy works for the variant game. So assume that the cats have a strategy $\mathcal{S}$ for the variant game on $G$. Let $t$ be the number of rounds of $\mathcal{S}$. The strategy for the standard game is as follows: first run $\mathcal{S}$. If $t$ is odd, now run $\mathcal{S}$ again. If $t$ is even then wait for one turn (i.e. on turn $t+1$ each cat chooses any vertex it likes) and now run $\mathcal{S}$ again. We claim that this strategy works for the cats in the standard game. Indeed, if the mouse starts on $V_1$ then they certainly catch it during the first run of $\mathcal{S}$. So we may assume that the mouse starts on $V_2$. But then as the graph is bipartite, immediately before every odd turn of the cats the mouse lies in $V_2$ and immediately before every even turn it lies in $V_1$. So if $t$ is odd then immediately before the the ($t+1$)th turn of the cats it lies on $V_1$ and we catch it during the second run of $\mathcal{S}$, whereas if $t$ is even then the waiting move ensure that immediately before we run $\mathcal{S}$ again the mouse lies on $V_1$ and therefore we subsequently catch it.  
\end{proof}

Before stating the main theorem of this section, let us take a closer look
at the proof of Lemma \ref{basicupperbound}. Given a tree $T$, we chose a
vertex $v$ of the tree so that $T-v$ consisted of two disjoint parts $T_1$
and $T_2$ not joined to each other by any edge, for which we could by
induction find winning strategies $\mathcal{S}_{T_1}, \mathcal{S}_{T_2}$ with
one less cat; therefore it sufficed to post one cat at $v$ while the other cats were running the winning strategies on $T_1$ and then $T_2$. But since the we obtain the winning strategies by induction, $\s_1$ itself consists in splitting $T_1$ into two smaller parts, call them $T_3$ and $T_4$ and running a winning strategy on each while guarding the vertex guarding them which we shall call $u$. So while two cats are guarding $u$ and $v$, the rest of the cats are running winning strategies on $T_3$ and then $T_4$. An important observation is as follows. If $u$ and $v$ do not belong to the same vertex class of $T$ then it is possible, in this variant of the game, to guard them both using \emph{one cat only}. Indeed, in this variant of the game the cats know at each step which class the mouse lies on, say the mouse starts in the class of $u$. Before each odd turn of the cats the mouse lies in the same class as $u$ and before each even turn on the same class as $v$. Therefore a cat choosing $u$ on each odd turn and $v$ on each even turn makes sure that the mouse never visits $u$ or $v$ while it is implementing this guarding strategy. This is the basis for our proof that in fact that essentially $(1/2)\log_2(n)$ cats are enough to win on any tree. Of course, there is the difficulty of having to guard two vertices in the same class at the same time - but it turns out that this problem can be circumvented. 

\begin{theorem}
Let $T$ be a tree of order $n$. Then $\lceil(1/2)\log(n)\rceil$ cats have a winning strategy on $T$. 
\end{theorem}  
 
\begin{proof}
The proof is by induction on $n$. It is clear by Theorem \ref{1catwin} that one cat can catch the mouse on trees of order $n\leq 9$. So let $T$ be a tree of order $n\geq 10$.  By Lemma \ref{centre} $T$ has a centre $v$. Let us denote by $V_1$ and $V_2$ its two vertex classes; we may assume without loss of generality that $v\in V_1$. Let $\mathcal{B}=\{B_1,B_2,\ldots, B_j \}$ be the set of components of order more than $n/4$ in $T-v$. Let $U$ be the union of the comonents of order no more than $n/4$ in $T-v$. Clearly there are at most $3$ components of size more than $n/4$ in $T-v$, i.e. $j\leq 3$. Let $r = \lceil (1/2)\log(n/4)\rceil$, so that $1+r \leq \lceil (1/2)\log(n)\rceil$. Let us assume for the moment that $j=3$ as this case is the most difficult to handle. 

For $i\in \{1,2,3 \}$ let $b_i$ be a centre for $B_i$; let $B_i^+$ be the component of $B_i - b_i$ which contains a vertex joined by an edge to $v$ and let $B_i^-$ be the union of the other components of $B_i - v$. We consider four different cases.    

\begin{description}
\item[Case 1]
All three of $b_1$, $b_2$ and $b_3$ belong to $V_2$. Since $B_i^-$ is a forest of trees each of order no more than $n/4$ we know by induction hypothesis that there exists a winning strategy $\s_{B_i^-}$ for $r$ cats on $B_i^-$. Likewise let $s_{B_i^+}$ be a winning strategy for $r$ cats on $B_i^+$ and let $\s_U$ be a winning strategy for $r$ cats on $U$. Clearly we may assume without loss of generality that all the strategies which we have just chosen last for an even number of turns. We now define a strategy $\s$ for the $|V_1|$-variant of the game for $r+1$ cats on $T$ as follows. One of the $r+1$ cats is called "the guard" and has a special role in the strategy.  The remaining cats will be referred to as "the soldiers". 
\begin{description}
\item[Stage 1]
The guard alternates between $v$ and $b_1$, starting with $v$. In the meantime the soldiers implement $\s_{B_1^+}$ on $B_1^+$ and then $\s_{B_1^-}$ on $B_1^-$.  
\item[Stage 2]
The guard alternates between $v$ and $b_2$, starting with $v$. In the meantime the soldiers implement $\s_{B_2^+}$ on $B_2^+$ and then $\s_{B_2^-}$ on $B_2^-$.   
\item[Stage 3]
The guard alternates between $v$ and $b_3$, starting with $v$. In the meantime the soldiers implement $\s_{B_3^+}$ on $B_3^+$ and then $\s_{B_3^-}$ on $B_3^-$.  
\item[Stage 4]
The guard remains on $v$ while the soldiers implement $\s_U$ on $U$.   
\end{description}
Let us now check that $\s$ is a winning strategy for the cats. Notice that since each stage lasts for an even number of turns, the mouse is caught if it ever visits $v$ (the guard is always on $v$ whenever the mouse lies on $V_1$). Therefore, it must start in $B_i$ for some $i$ or in $U$ and remain in this set for the duration of the game. Suppose it starts in $B_i$ for $i\in \{1,2,3\}$. Then at the beginning of stage $i$ it lies on a vertex in $V_1$ (hence not $b_i$) and is trapped in either one of $B_i^+$ or $B_i^-$ for the duration of this stage. But then it is clearly caught by the soldiers at some point during this stage. If the mouse started the game in $U$ then clearly it will be caught during stage 4. 
\item[Cases 2 and 3]
One or two of the $b_i$'s belong to $V_1$ and one belongs to $V_2$. This case will be left to the reader, as it is simpler than Case 4 below.  
\item[Case 4]
Each one of $b_1$, $b_2$ and $b_3$ lies in $V_1$. In this case we let $b'_1$ be that vertex of $B_1^+$ joined to $b_1$ by an edge and let $B_1^*$ be $B_1^+ - b'_1$. $B_2^*$ and $b'_2$ are likewise defined. However, for the definition of $b'_3$ and $B_3^*$ is slightly different: $b'_3$ is the element of $B_3$ joined by an edge to $v$ ) and $B_3^*$ denotes $B_3 - b'_3$. As in Case 1, we let $\s_X$ denote a winning strategy for $r$ cats on the forest $X$ consisting of trees each of order no more than $n/4$. Furthermore, all such strategies that we choose shall last for an even number of turns. We define a strategy $\s$ for $r+1$ cats as follows. We shall still have one guard and $r$ soldiers. 
\begin{description}
\item[Stage 1]
The guard remains on $b_1$ throughout this stage. The soldiers implement $\s_{B_1^-}$ on $B_1^-$. 
\item[Stage 2]
This stage lasts two turns. The guard chooses $b_1'$ for one turn while the soldiers wait. On the second turn the guard chooses $b_1'$ while the soldiers wait. 
\item[Stage 3]
The guard alternates between $v$ and $b_1'$, starting with $v$. In the meantime the soldiers implement $\s_{B_1^*}$ on $B_1^*$. 
\item[Stage 4]
The guard alternates between $v$ and $b'_2$, starting with $v$. In the meantime, the soldiers implement $\s_{B_2^*}$ on $B_2^*$. 
\item[Stage 5]
The guard alternates between $v$ and $b'_2$, starting with $v$. The soldiers implement $\s_U$ on $U$. 
\item[Stage 6]
This stage lasts two turns. On the first turn, the guard chooses $v$. On the second turn the guard chooses $b'_2$, one of the soldiers chooses $b'_3$ while the remaining soldiers wait (we may clearly assume there is at least one soldier). 
\item[Stage 7]
The guard alternates between $b'_3$ and $b_2$. In the meantime, the soldiers implement $\s_{B_2^-}$ on $B_2^-$. 
\item[Stage 8]
The guard alternates between $b'_3$ and $b_3$. In the meantime, the soldiers implement $s_{B_3^*}$ on $B_3^*$ and then $B_3^-$ on $B_3^-$.     
\end{description}
We shall now check that $\s$ is a winning strategy for the cats. Notice that since each stage of strategy implemented by the cats lasts for an even number of turns, we know that at the beginin of each such stage (that is, immediately before the first turn of the cats), the mouse lies in $V_1$. If the mouse starts the game in $B_1^-$ then it is clearly caught during Stage 1, since the guard ensures that it cannot escape this set. Thus at the beginning of stage $2$, we may assume that the mouse lies in $(T- B_1^-)\cap V_1$. Then it is clear that at the beginning of stage 3 that the mouse must lie in $(T - B_1^- - b_1)\cap V_1 $ if it wasn't already caught. During Stage 3 it may never travel to either $v$ or $b'_1$ thanks to the guard. Thus, if at the beginning of this stage it lied in $B_1^* \cup \{b'_1\}$, it was caught by the cats. Otherwise, it began this stage in $T - B_1 - v$ and remained there throughout this stage. Thus, we may assume that at the beginning of stage 4 the mouse lied in $T - B_1 - v$. Throughout this stage, thanks to the guard, the mouse never visits $v$ and hence in particular never enters $B_1$. Moreover, it is clear that if it started this stage in $B_2^*$ then it was caught by the cats. Thus at the beginning of Stage 5 the mouse lied in $(T - B_1 - B_2^*)\cap V_1$. During stage $5$ both $v$ and $b'_2$ are guarded by the guard, so the mouse never entered $B_1$ or $B_2^*$. Moreover, it is clear that if the mouse started this stage in $U$ then it was caught by the soldiers. Therefore it started in $B_2^- \cup \{b_2\} \cup B_3$ and never escaped this set during stage 5; therefore we may assume that at the beginning of stage 6 it lied in $(B_2^*\cup B_3\cup\{v\})\cap V_1$. Clearly after the end of Stage 6, it lied in $(B_2^*\cup B_3)\cap V_1$. As the guard now alternates between $b_3'$ and $b_2$, the mouse is caught during stage 7 if it lied in $B_2^-$ at the beginning of this stage. Thus we may assume that it started in $B_3$ and never escaped this set since the guard was guarding $b'_3$. Now in stage $8$ it is clear that the cats will catch the mouse.         
\end{description} 

The cases where $j \le 2$ can be treated in a similar manner.
\end{proof}

\section{Lower bound}

In this section we establish the lower bound for the number of cats
on a tree of given order. More precisely, for each $k \ge 1$ we construct
a tree $T_k$ on $2^{k+2} - 3$ vertices for which $(1/4 - o(1)) k$ cats
are necessary to catch the mouse. This implies that
$g(n) \ge (1/4 - o(1)) \log_2 n$.

We now describe the construction of the trees in question. For any
integer $k \ge 1$, let $B_k$ be the complete binary tree of height $k$.
Then define $T_k$ to be the $1$-subdivision of $B_k$, that is the tree
obtained from $B_k$ by replacing every edge with a path of length two.
Note that $B_k$ has $2^{k+1}-1$ vertices and one fewer edges, and so
$|T_k| = 2^{k+2} - 3$.

We shall first prove that the number of cats required to catch the mouse
on $T_k$ is $\Omega(k)$. Afterwards we shall refine the argument to
prove that in fact this number is at least $(1/4 - o(1)) k$, implying
Theorem~\ref{lowerbound}.

\subsection{A linear bound}

Here we prove the weaker lower bound $g(n) = \Omega(\log n)$.

We start with some preliminary lemmas. The first one is regarding
arithmetic properties of natural numbers. Given a natural number $n$,
we define $\gamma(n)$ as the number of positions in the binary expansion
of $n$ where a one is followed by a zero. In other words, if
$n = \sum_{i=0}^r a_i 2^i$ for some $a_i \in \{0, 1\}$,
then $\gamma(n) = |\{i \in \{1, \dotsc, r\} : a_i = 1,\,a_{i-1} =0\}|$.

Moreover, we define $\beta(n)$ as the least integer $r \ge 1$ such
that $n$ can be expressed as the sum $n = \sum_{i=1}^r s_i 2^{a_i}$
with $s_i \in \{1, -1\}$ and $a_i \in \N_0$.
In other words, $\beta(n)$ is the least possible number of terms
in an expression of $n$ as the sum and difference of powers of two.
For consistency we shall define $\beta(0) = \gamma(0) = 0$.

These two functions satisfy a relation.

\begin{lem} \label{lem:arithmetic}
    For any positive integer $n$, $\beta(n) \ge \gamma(n)$.
\end{lem}

\begin{proof}
    Suppose that $n$ can be expressed as a sum
    $\sum_{i=1}^r 2^{a_i} - \sum_{i=1}^s 2^{b_i}$, where $r, s, a_i$
    and $b_i$ are non-negative integers. Our aim is to prove that
    $\gamma(n) \le r + s$. We shall achieve this by induction on $s$.

    If $s = 0$ then $n$ has at most $r$ ones in its binary expansion,
    so clearly $\gamma(n) \le r$. Now suppose that $s \ge 1$. We may
    assume that the $a_i$'s and $b_i$'s are all distinct and
    write $b = \min \{ b_i : 1 \le i \le s \}$. Note that
    $n = n' - 2^b$ where $\gamma(n') \le r + s - 1$ by the induction
    hypothesis.

    For any positive integer $m$, let $(m)_i$ denote the $i$-th lowest
    digit in the binary expansion of $m$.
    By the choice of $b$, $(n')_b = 0$: a one in this position could only
    come from a positive summand $2^{a_i}$, but no $a_i$ equals $b$ by
    assumption. Since $n$ is positive, $(n')_a = 1$ for some $a \ge b$.
    Take the least such $a$, so that $(n')_i = 0$ for every
    $b \le i \le a - 1$. Now $n$ differs from $n'$ in exactly
    the digits $b$ through $a$: $(n)_a = 0$ and $(n)_i = 1$ for every
    $b \le i \le a - 1$.
    \begin{align*}
        n' \,=\, &\texttt{\dots1000000\dots} \\
        n  \,=\, &\texttt{\dots0111111\dots}
    \end{align*}

    Observe that $\gamma(n) - \gamma(n')$ is maximised if the
    $b$-th digit is succeeded by a zero and the $a$-th digit is
    preceeded by a one. In such case this difference is equal
    to one. Therefore in general $\gamma(n) \le \gamma(n') + 1
    \le r + s$, proving the claim.
\end{proof}

Recall that $T_k$ is the $1$-subdivision of the complete binary tree on
$k$ levels $B_k$. We shall distinguish the vertices originally present
in $B_k$ by calling them \emph{important}. For any $X \subset V(T_k)$,
we define $\partial X$ as the set of important vertices not present
in $X$, but which are distance two away from an important vertex in $X$.
In other words, $\partial X$ is the vertex boundary of $X \cap V(B_k)$
in $B_k$.

\begin{lem} \label{lem:weakboundary}
    Let $X \subset V(T_k)$ and let $n$ be the number of important
    vertices contained in $X$. Then $|\partial X| \ge
    (\gamma(n) - 2)/6$.
\end{lem}

\begin{proof}
    Let $m = |\partial X|$. In the view of Lemma~%
    \ref{lem:arithmetic}, it is enough to prove that $\beta(n) \le 6m + 2$.
    To achieve this we shall use induction on $m$.
    
    We may assume that $X$ contains only important vertices. Then we shall
    consider $X$ as a vertex subset of $B_k$, and $\partial X$ as its
    vertex boundary.

    If $m = 0$, then either $X = V(B_k)$ or $X = \es$. Then $n = |X|$ is
    either $2^{k+1} - 1$ or $0$, so $\beta(n) \le 2$.  
    Now suppose that $m \ge 1$. We shall split the argument into two
    cases, depending on whether $X$ contains the root of $B_k$.

    \begin{description}
        \item[Case 1:] $X$ contains the root of $B_k$.
            
            Let
            $x_1, \dotsc, x_t$ be the maximal elements of $\partial X$,
            meaning that they are the elements of $\partial X$ without
            an ancestor contained in $\partial X$.

            Let us focus on 
            a particular $x_i$. First, let $k_i$ denote the distance from
            $x_i$ to the nearest leaf. Write $S_1$ and $S_2$ for the subtrees
            rooted at the children of $x_i$ (if $x_i$ is a leaf, then
            set $S_1 = S_2 = \es$). Write $n_i$ and $m_i$ for the number
            of elements of respectively $X$ and $\partial X$ contained
            in $V(S_1) \cup V(S_2)$. By the induction hypothesis,
            $\beta(n_i) \le 6 m_i + 4$.

            Now $n = 2^{k+1} - 1 - \sum_{i=1}^t (2^{k_i+1} - 1 - n_i)$,
            so
            \begin{align*}
                \beta(n) &\,\le\, 2(t+1) + \sum_{i=1}^t \beta(n_i) \\
                         &\,\le\, 6t + 2 + 6 \sum_{i=1}^t m_i \\
                         &\,=\, 6m + 2.
            \end{align*}

        \item[Case 2:] $X$ does not contain the root of $B_k$.

            In this case (using the same notation as before)
            $n = \sum_{i=1}^t n_i$. Therefore $\beta(n) \le 
            \sum_{i=1}^t \beta(n_i) \le 6(m-t) + 2 < 6m + 2$. \qedhere
    \end{description}
\end{proof}

\begin{lem} \label{lem:weakmain}
    Let $n \le 2^{k+1} - 1$. If the number of cats playing on $T_k$ is 
    at most $(\gamma(n) - 2)/18$, then the mouse can indefinitely avoid being
    captured.
\end{lem}

\begin{proof}
    Suppose the number of cats is $m \le (\gamma(n) - 2)/18$.
    Let $C_1, C_2, \dots$ be a sequence of $m$-sets of $V(T_k)$,
    indicating a strategy for the cats. More precisely, at each step $i$,
    the cats will occupy the vertices in $C_i$.

    Let $A_0, A_1, A_2, \dots$ be the sequence of sets, representing
    the possible mouse positions at each step. In other words,
    $A_0 = V(T_k)$ and $A_{i+1}$ is obtained by taking the vertex boundary
    of $A_i$ in $T_k$ and removing $C_{i+1}$ from it. We shall show that
    for any integer $i \ge 0$, the set $A_{2i}$ contains at least $n$
    important vertices, and in particular is not empty.

    Once again we proceed by induction. If $i = 0$, then $A_0$ contains
    $2^{k+1}-1$ important vertices. Now suppose that $i \ge 2$ and assume
    that $A_{2(i-1)}$ contains at least $n$ important vertices. Let $A$
    be a subset of $A_{2(i-1)}$, consisting of exactly $n$ important
    vertices. Then $|\partial A| \ge (\gamma(n) - 2)/6$ by Lemma~%
    \ref{lem:weakboundary}.

    Let $B$ denote the vertex boundary of $C_{2i-1}$ in $T_k$. Then $A_{2i}$
    contains
    $$
        (A \cup \partial A) \sm \left( B \cup C_{2i}\right),
    $$
    and this set has size at least $|A| + |\partial A| - 3m \ge n$.
\end{proof}

\begin{cor}
    Let $k \ge 100$ and suppose that the number of cats playing on
    $T_k$ does not exceed $k / 40$. Then the mouse can indefinitely avoid
    being captured.
\end{cor}

\begin{proof}
    In order to apply Lemma~\ref{lem:weakmain}, it suffices to find
    $n \le 2^{k+1} - 1$ for which $\gamma(n) \ge \lfloor k / 2 \rfloor$.
    Clearly,
    one such number is given by $n = 101010\dots10_2$ with
    $\lfloor k / 2 \rfloor$ repeats of $10_2$.
\end{proof}

\subsection{Improving the lower bound}

In this subsection we refine the previous argument to prove that
$(1/4 + o(1)) k$ cats are necessary to catch the mouse on $T_k$.
In particular, we prove Theorem~\ref{lowerbound}.

\begin{lem} \label{lem:approximate}
    Let $m, n$ and $k$ be positive integers with $|m - n| \le 2^k$.
    Then $|\beta(m) - \beta(n)| \le k$.
\end{lem}

\begin{proof}
    We may assume that $m \ge n$ and then $m - n$ can be expressed
    as a sum of at most $k$ powers of two.
\end{proof}

\begin{lem}
    For any $\epsilon > 0$ there is some $k_0  \ge 1$ with the following
    property. Let $k \ge k_0$, $X \subset V(T_k)$ and let $n$ be the
    number of important vertices contained in $X$. Then $|\partial X|
    \ge \gamma(n) - \epsilon k$.
\end{lem}

\begin{proof}
    Take any $k \ge 1$ and any $X \subset V(T_k)$. We may assume
    that $X$ contains $n$ important vertices and nothing else. This
    allows us to consider $X$ as a subset of $V(B_k)$ and $\partial X$
    as its vertex boundary in $B_k$. Write $|\partial X| = m$. We
    will prove by induction on $m$ that $n = n' + d$ for some integers
    $n' \ge 0$ and $d$ such that $\beta(n') \le m + 2$ and $|d| \le 2m$.

    If $m = 0$, then $n = 0$ or $n = 2^{k+1} - 1$, and the claim is true.
    Now suppose that $m \ge 1$ and let $x$ be a minimal element of 
    $\partial X$. Let $S$ be the subtree of $B_k$ rooted at $x$;
    because $x$ is minimal, $S$ does not contain other elements of
    $\partial X$. The argument splits into two cases.
    
    \begin{description}
        \item[Case 1:] $x$ is the root of $B_k$, or the parent of $x$
            is not in $X$.

            Let $X' = X \sm S$. Then $\partial X' = (\partial X) \sm \{x\}$,
            so by induction hypothesis $|X'| = t + r$ with $\gamma(t)
            \le m + 1$ and $|r| \le 2m - 2$. Now it suffices to note
            that $|X| - |X'|$ is either $2^a - 1$ or $2^{a+1} - 2$ for
            some $a \ge 1$.

        \item[Case 2:] the parent of $x$ is in $X$.

            This time let $X' = X \cup S$. Again we have $\partial X'
            = (\partial X) \sm \{x\}$, so we can apply induction
            in the same way, but this time noting that $|X'| - |X|$
            is either $1$ or $2^a$, or $2^{a+1} - 1$ for some $a \ge 1$.
    \end{description}

    By Lemmas~\ref{lem:arithmetic}~and~\ref{lem:approximate} we now have
    $\gamma(n) \le \beta(n) \le \beta(n') + \lceil \log_2 |d| \rceil \le m
    + \log_2 m + 4$. Now fix $\epsilon > 0$. Note that $n \le |B_k| =
    2^{k+1}-1$, so $\gamma(n) \le k$. Therefore if $|\partial X| = m \ge k$,
    then we trivially have $|\partial X| \ge \gamma(n) - \epsilon k$. Now
    assume that $m \le k$, so $m \ge \gamma(n) - \log_2 k - 4$. We are done
    by selecting $k_0$ such that $\log_2 k + 4 \le \epsilon k$ for all
    $k \ge k_0$.
\end{proof}

\begin{lem}
    For any $\epsilon > 0$ there is some $k_0 \ge 1$ such that if at most
    $(1/4 - \epsilon) k$ cats are playing on $T_k$ for $k \ge k_0$, then
    the mouse can indefinitely avoid being captured.
\end{lem}

\begin{proof}
    Let $\epsilon > 0$ and fix sufficiently large $k_0$ so that any following
    assumptions about its size will hold.
    Suppose that $k \ge k_0$ and let $c$ cats be playing on $T_k$, where
    $c \le (1/4 - \epsilon) k$. Denote by $n$ the number whose binary
    representation consists of $\lfloor k/2 \rfloor$ repeats of $10$, that
    is, $n = 2 \sum_{i=0}^{\lfloor k/2 \rfloor - 1} 4^i =
    2(4^{\lfloor k/2 \rfloor} - 1)$. Observe that $n \le 2^{k+1} - 1$
    and that $\gamma(k) = \lfloor k/2 \rfloor$.

    Let $C_1, C_2, \dots$ be a cat strategy, that is, a sequence of
    $c$-sets from $V(T_k)$. Let $A_0, A_1, \dots$ be the corresponding
    sequence of possible mouse positions, that is, $A_0 = V(T_k)$ and
    $A_{i+1}$ is obtained by taking the vertex boundary of $A_i$ in $T_k$
    and removing $C_{i+1}$ from it. We will prove by induction that
    for each $i \ge 0$, $A_{2i}$ contains at least $n$ important vertices.

    The claim is clearly true for $i = 0$, so assume that $i \ge 1$
    and let $X$ be a subset of $A_{2(i-1)}$ consisting of exactly $n$
    important vertices. Consider $X$ as a subset of $V(B_k)$ and denote
    $\X = X \cup \partial X$. Let $X_1, \dotsc, X_t$ be the connected
    components of $\X$ in $B_k$. Observe that since $X \neq V(B_k)$,
    each $X_j$ contains an element of $\partial X$, and hence $t \le
    |\partial X|$.

    Let us for the moment focus on a single component $X_j$. Let $E_j$
    be the set of edges of $B_k[X_j]$ whose subdividing vertices
    belong to $C_{2i-1}$ (that is, at the ($2i-1$)-th step cats play
    on the vertices of $T_k$ corresponding to these edges of $B_k$).
    Also, let $Y_j = C_{2i} \cap V(X_j)$ and let $R_j$ be the set
    of vertices of $X_j$ whose all incident edges in $B_k[X_j]$
    belong to $E_j$. Note that
    $$
        A_{2i} \supset V(X_j) \sm (Y_j \cup R_j).
    $$
    Our goal now is to estimate $|R_j|$ in terms of $|E_j|$ and 
    $|V(X_j)|$.

    Observe that $R_j$ is the set of isolated vertices in $X_j \sm E_j$.
    In particular, if $X_j \sm E_j$ has exactly $r$ connected components,
    then $|R_j| \le r$. In fact, equality can hold only if $X_j \sm E_j$
    is empty, so if $|E_j| < |V(X_j)| - 1$ then $|R_j| \le r-1 \le |E_j|$.
    We can put this more coincisely as $|R_j| \le |E_j|+ w_j$ where
    $w_j = 1$ if $|E_j| = |V(X_j)| - 1$ and $w_j = 0$ otherwise.

    At this stage we know that the number of important vertices
    contained in $A_{2i}$ is at least
    \begin{align*}
        &\,\sum_{i=1}^t \left( |V(X_j)| - |Y_j| - |R_j| \right ) \\
        \ge&\,|X| + |\partial X| - c - \sum_{j=1}^t \left(|E_j| + w_j\right) \\
        \ge&\,n + |\partial X| - 2c - \sum_{j=1}^t w_j. \\
    \end{align*}

    Let $I$ denote the set of indices $j \in [t]$ such that $w_j = 1$
    and write
    $X' = X \sm \bigcup_{j\in I} X_j$. Observe that $|I| \le c$ 
    (because $\X$ cannot have isolated vertices) and that for any $j \in I$
    we have $|V(X_j)| \le c + 1 \le 2c$. Therefore $|X'| \ge |X| - 2c^2
    \ge |X| - k^2$. By choosing large enough $k_0$ we can guarantee
    \begin{align*}
        |\partial X| &\ge |\partial X'| + |I| \\
                     &\ge \gamma(n) - 2 \lceil \log_2 k \rceil + |I| \\
                     &\ge \left(\frac{1}{2}-2\epsilon\right)k + |I|.
    \end{align*}
    Combining this with the previous inequality, we can conclude that
    $A_{2i}$ contains at least $n$ important vertices.

    We have proved that each of the sets $A_0, A_2, \dotsc$ contains
    at least $n$ important vertices, so they are all not empty.
    Therefore the mouse can indefinitely avoid being captured.
\end{proof}

\end{document}